\documentclass{amsart}
\usepackage{amssymb}

\usepackage{amscd}
\usepackage{thmdefs}

\input{tcilatex}
\theoremstyle{definition}
\theoremstyle{remark}
\numberwithin{equation}{section}

\input tcilatex

\begin{document}
\title[Weighted Sobolev spaces with nontrivial power weights ]{Special embeddings of weighted Sobolev spaces with nontrivial power weights}
\author{Patrick J. Rabier}
\address{Department of Mathematics, University of Pittsburgh, Pittsburgh, PA 15260}
\email{rabier@imap.pitt.edu}
\maketitle

\begin{abstract}
In prior work, the author has characterized the real numbers $a,b,c$ and $%
1\leq p,q,r<\infty $ such that the weighted Sobolev space $%
W_{\{a,b\}}^{(q,p)}(\Bbb{R}^{N}\backslash \{0\}):=\{u\in L_{loc}^{1}(\Bbb{R}%
^{N}\backslash \{0\}):|x|^{\frac{a}{q}}u\in L^{q}(\Bbb{R}^{N}),|x|^{\frac{b}{%
p}}\nabla u\in (L^{p}(\Bbb{R}^{N}))^{N}\}$ is continuously \thinspace
embedded \thinspace into $L^{r}(\Bbb{R}^{N};|x|^{c}dx)$\noindent $:=\{u\in
L_{loc}^{1}(\Bbb{R}^{N}\backslash \{0\}):|x|^{\frac{c}{r}}u\in L^{r}(\Bbb{R}%
^{N})\}.$

This paper discusses the embedding question for $W_{\{a,b\}}^{(\infty ,p)}(%
\Bbb{R}^{N}\backslash \{0\}):=\{u\in L_{loc}^{1}(\Bbb{R}^{N}\backslash
\{0\}):|x|^{a}u\in L^{\infty }(\Bbb{R}^{N}),|x|^{\frac{b}{p}}\nabla u\in
(L^{p}(\Bbb{R}^{N}))^{N}\},$ which is not the space obtained by the formal
substitution $q=\infty $ in the previous definition of $W_{\{a,b\}}^{(q,p)}(%
\Bbb{R}^{N}\backslash \{0\}),$ unless $a=0.$

The corresponding embedding theorem identifies all the real numbers $a,b,c$
and $1\leq p,r<\infty $ such that $W_{\{a,b\}}^{(\infty ,p)}(\Bbb{R}%
^{N}\backslash \{0\})$ is continuously embedded in $L^{r}(\Bbb{R}
^{N};|x|^{c}dx).$ A notable feature is that such embeddings exist only when $%
a\neq 0$ and, in particular, have no analog in the unweighted setting.

It is also shown that the embeddings are always accounted for by
multiplicative rather than just additive norm inequalities. These
inequalities are natural extensions of the Caffarelli-Kohn-Nirenberg
inequalities which, in their known form, are restricted to functions of $%
C_{0}^{\infty }(\Bbb{R}^{N})$ and do not incorporate supremum norms.
\end{abstract}

\section{Introduction\label{intro}}

Throughout this paper, $\Bbb{R}_{*}^{N}:=\Bbb{R}^{N}\backslash \{0\}.$ Given 
$d\in \Bbb{R},$ the measure $|x|^{d}dx$ on $\Bbb{R}_{*}^{N}$ can be extended
to a measure on $\Bbb{R}^{N}$ provided that the $|x|^{d}dx$-measure of $%
\{0\} $ is defined to be $0$ (this must be specified if $d\leq -N$). If so,
the space $L^{s}(\Bbb{R}^{N};|x|^{d}dx),0<s<\infty ,$ coincides with the
space of Lebesgue measurable functions $u$ on $\Bbb{R}^{N}$ such that $|x|^{%
\frac{d}{s}}u\in L^{s}(\Bbb{R}^{N}).$ The norm (quasi-norm when $0<s<1$) of $%
u\in L^{s}(\Bbb{R}^{N};|x|^{d}dx)$ will be denoted by $||u||_{d,s}:=||\,|x|^{%
\frac{d}{s}}u||_{s},$ where $||\cdot ||_{s}$ is the (quasi) norm of $L^{s}(%
\Bbb{R}^{N}).$

If $a,b\in \Bbb{R}$ and $1\leq p<\infty $ and $0<q<\infty ,$ set 
\begin{multline}
W_{\{a,b\}}^{1,(q,p)}(\Bbb{R}_{*}^{N}):=  \label{1} \\
\{u\in L_{loc}^{1}(\Bbb{R}_{*}^{N}):u\in L^{q}(\Bbb{R}^{N};|x|^{a}dx),\quad
\nabla u\in (L^{p}(\Bbb{R}^{N};|x|^{b}dx))^{N}\},
\end{multline}
with (quasi) norm $||u||_{a,q}+||\nabla u||_{b,p}$. Note that $%
W_{\{a,b\}}^{1,(q,p)}(\Bbb{R}_{*}^{N})$ is \emph{not} defined as the
-usually smaller and unknown- closure of some subspace of smooth enough
functions.

Recently, the author has characterized all the real numbers $a,b,c$ and $%
1\leq p,q,r<\infty $ ($1\leq p<\infty $ and $0<q,r<\infty $ if $N=1$) such
that $W_{\{a,b\}}^{1,(q,p)}(\Bbb{R}_{*}^{N})\hookrightarrow L^{r}(\Bbb{R}
^{N};|x|^{c}dx)$ where, as usual, ``$\hookrightarrow $'' refers to
continuous embedding, and shown that, with a single exception, this
embedding is accounted for by a multiplicative inequality (\cite{Ra11}).
This generalizes both the Sobolev embedding theorem in the unweighted case $%
a=b=c=0$ and the Caffarelli-Kohn-Nirenberg (CKN) inequalities \cite{CaKoNi84}%
\ when $a,b,c>-N$ and $u\in C_{0}^{\infty }(\Bbb{R}^{N}).$

If $W_{\{a,b\}}^{1,(q,p)}(\Bbb{R}_{*}^{N})\hookrightarrow L^{r}(\Bbb{R}
^{N};|x|^{c}dx),$ then $W_{\{a,b\}}^{1,(q,p)}(\Bbb{R}^{N})\hookrightarrow
L^{r}(\Bbb{R}^{N};|x|^{c}dx)$ since a smaller space is obtained when $\Bbb{R}%
_{*}^{N}$ is replaced by $\Bbb{R}^{N}$ in (\ref{1}). However, it is only
with $\Bbb{R}_{*}^{N}$ that the admissible values of the parameters $%
a,b,c,p,q,r$ have been exactly identified. If $b\leq 0,$ these admissible
values are the same whether $\Bbb{R}^{N}$ or $\Bbb{R}_{*}^{N}$ is used. This
is essentially trivial if $N=1$ and due to $W_{\{a,b\}}^{1,(q,p)}(\Bbb{R}
^{N})=W_{\{a,b\}}^{1,(q,p)}(\Bbb{R}_{*}^{N})$ if $b\leq 0$ and $N\geq 2$
(Remark \ref{rm4}).

The goal of the present paper is to show that, in a suitable form, the
results of \cite{Ra11} can be extended when $p$ and $r$ are still finite
but, \emph{roughly speaking,} $q=\infty ,$ although the problem is trivial
if this statement is taken literally. Indeed, since $L^{\infty }(\Bbb{R}%
^{N};|x|^{a}dx)=L^{\infty }(\Bbb{R}^{N})$ is independent of $a,$ the
constant function $u=1$ gives an example when $u\in L^{\infty }(\Bbb{R}
^{N};|x|^{a}dx)$ and $\nabla u\in (L^{p}(\Bbb{R}^{N};|x|^{b}dx))^{N}$
irrespective of $a,b,p,$ yet $u$ does not belong to $L^{r}(\Bbb{R}
^{N};|x|^{c}dx)$ for any $c\in \Bbb{R}$ and $0<r<\infty $. Thus, no
embedding is true when $q=\infty $ in (\ref{1}).

The fact that the substitution $q=\infty $ in (\ref{1}) produces a space
independent of $a$ suggests that this is not how $W_{\{a,b\}}^{1,(q,p)}(\Bbb{%
R}_{*}^{N})$ should be defined when $q=\infty .$ As we shall see, the
``correct'' definition is given by 
\begin{multline}
W_{\{a,b\}}^{1,(\infty ,p)}(\Bbb{R}_{*}^{N}):=  \label{2} \\
\{u\in L_{loc}^{1}(\Bbb{R}_{*}^{N}):|x|^{a}u\in L^{\infty }(\Bbb{R}
^{N}),\quad \nabla u\in (L^{p}(\Bbb{R}^{N};|x|^{b}dx))^{N}\},
\end{multline}
equipped with the natural norm 
\begin{equation}
||\,|x|^{a}u||_{\infty }+||\nabla u||_{b,p}.  \label{3}
\end{equation}

As in the case of (\ref{1}) when $q<\infty ,$ the space $W_{\{a,b\}}^{1,(%
\infty ,p)}(\Bbb{R}_{*}^{N})$ contains the space $W_{\{a,b\}}^{1,(\infty
,p)}(\Bbb{R}^{N})$ obtained by replacing $\Bbb{R}_{*}^{N}$ by $\Bbb{R}^{N}$
in (\ref{2}). Once again, both spaces coincide when $N\geq 2$ and $b\leq 0;$
see Remark \ref{rm4}.

Unlike $|x|^{a}u\in L^{q}(\Bbb{R}^{N})$ with $q<\infty ,$ which is just $%
u\in L^{q}(\Bbb{R}^{N};|x|^{aq}dx),$ when $a\neq 0$ the condition $%
|x|^{a}u\in L^{\infty }(\Bbb{R}^{N})$ is not of the type $u\in L^{s}(\mu )$
for some measure $\mu $ on $\Bbb{R}^{N}$ or $\Bbb{R}_{*}^{N}$ and some $%
0<s\leq \infty .$ On the other hand, the space obtained by setting $q=\infty 
$ in (\ref{1}) is recovered if and only if $a=0$ in (\ref{2}). Accordingly, $%
a\neq 0$ is necessary for any embedding. In particular, the results of this
paper do not generalize a property already familiar in classical,
unweighted, Sobolev spaces. They are special to weighted spaces with
nontrivial (power) weights, although $b=0$ is not ruled out.

Some notation must be introduced for the statement of the embedding Theorem 
\ref{th1} below, whose proof is the single purpose of this paper. As is
customary, if $1\leq p<\infty ,$ then 
\begin{equation*}
p^{*}=\infty \text{ if }p\geq N\text{ and }p^{*}=\frac{Np}{N-p}\text{ if }%
1\leq p<N.
\end{equation*}
Next, we denote by $c^{0}$ and $c^{1}$ the two points 
\begin{equation}
c^{0}:=ar-N\qquad \text{and\qquad }c^{1}:=\frac{r(b-p+N)}{p}-N,  \label{4}
\end{equation}
where it is understood that $a,b,p$ and $r$ are given. The points $c^{0}$
and $c^{1}$ are distinct if and only if $ap-N\neq b-p.$ If so and if $c$ is
in the closed interval with endpoints $c^{0}$ and $c^{1},$ we set 
\begin{equation}
\theta _{c}:=\frac{c-c^{0}}{c^{1}-c^{0}},  \label{5}
\end{equation}
so that $\theta _{c}\in [0,1]$ and that 
\begin{equation}
c=\theta _{c}c^{1}+(1-\theta _{c})c^{0}.  \label{6}
\end{equation}
In particular, $\theta _{c^{0}}=0$ and $\theta _{c^{1}}=1$ and, by (\ref{4}%
), (\ref{5}) and (\ref{6}), 
\begin{equation}
\frac{c+N}{r}=\theta _{c}\frac{b-p+N}{p}+(1-\theta _{c})a.  \label{7}
\end{equation}

\begin{theorem}
\label{th1}Let $a,b,c\in \Bbb{R}$ and $1\leq p,r<\infty $ be given ($1\leq
p<\infty $ and $0<r<\infty $ if $N=1$). Then, $W_{\{a,b\}}^{1,(\infty ,p)}(%
\Bbb{R}_{*}^{N})\hookrightarrow L^{r}(\Bbb{R}^{N};|x|^{c}dx)$ (and hence $%
W_{\{a,b\}}^{1,(\infty ,p)}(\Bbb{R}_{*}^{N})\hookrightarrow
W_{\{c,b\}}^{1,(r,p)}(\Bbb{R}_{*}^{N})$) if and only if $a\neq 0$ and one of
the following four conditions holds:\newline
(i) $ap-N$ and $b-p$ are on the same side of $-N$ (including $-N$), $%
ap-N\neq b-p,$ $c$ is in the open interval with endpoints $c^{0}$ and $c^{1}$
and $\theta _{c}\leq \frac{p^{*}}{r}.$\newline
(ii) $ap-N$ and $b-p$ are strictly on opposite sides of $-N$ (hence $%
ap-N\neq b-p$), $c$ is in the open interval with endpoints $c^{0}$ and $-N$
and $\theta _{c}\leq \frac{p^{*}}{r}.$\newline
(iii) $p\leq r\leq p^{*},a(b-p+N)>0$ (i.e., $ap-N$ and $b-p$ are strictly on
the same side of $-N$) and $c=c^{1}.$\newline
(iv) $p<N,r>p^{*},ap-N=b-p$ and $c=c^{1}$ ($=c^{0}$).\newline
Furthermore, when the embedding $W_{\{a,b\}}^{1,(\infty ,p)}(\Bbb{R}%
_{*}^{N})\hookrightarrow L^{r}(\Bbb{R}^{N};|x|^{c}dx)$ holds, it is always
characterized by a multiplicative inequality. Specifically:\newline
(v) If $a\neq 0$ and one of the conditions (i) to (iii) holds with $ap-N\neq
b-p$ (which is already assumed in (i) or (ii)), there is a constant $C>0$
such that 
\begin{equation}
||u||_{c,r}\leq C||\nabla u||_{b,p}^{\theta _{c}}||\,|x|^{a}u||_{\infty
}^{1-\theta _{c}},\qquad \forall u\in W_{\{a,b\}}^{1,(\infty ,p)}(\Bbb{R}%
_{*}^{N}),  \label{8}
\end{equation}
where $\theta _{c}$ is given by (\ref{4}) and (\ref{5}).\newline
(vi) If $a\neq 0,ap-N=b-p$ and if $p\leq r\leq p^{*},c=c^{1}$ ($=c^{0}$),
there is a constant $C>0$ such that 
\begin{equation}
||u||_{c^{1},r}\leq C||\nabla u||_{b,p},\qquad \forall u\in
W_{\{a,b\}}^{1,(\infty ,p)}(\Bbb{R}_{*}^{N}).  \label{9}
\end{equation}
(vii) If $a\neq 0,$ $p<N,ap-N=b-p$ and if $r>p^{*},c=c^{1}$ ($=c^{0}$),
there is a constant $C>0$ such that 
\begin{equation}
||u||_{c^{1},r}\leq C||\nabla u||_{b,p}^{\frac{p^{*}}{r}}||\,|x|^{a}u||_{%
\infty }^{1-\frac{p^{*}}{r}},\qquad \forall u\in W_{\{a,b\}}^{1,(\infty ,p)}(%
\Bbb{R}_{*}^{N}).  \label{10}
\end{equation}
\end{theorem}

The value $r=\infty $ is not included in Theorem \ref{th1}. This case
requires a different treatment and will be discussed elsewhere in a more
general framework (\cite{Ra11b}). It is plain that the restriction $\theta
_{c}\leq \frac{p^{*}}{r}$ in (i) and (ii) is only relevant when $p<N$ and $%
r>p^{*}.$

To avoid misunderstandings, it should be stressed that if $u$ is a
distribution on $\Bbb{R}^{N}$ whose restriction to $\Bbb{R}_{*}^{N}$ is in $%
W_{\{a,b\}}^{1,(\infty ,p)}(\Bbb{R}_{*}^{N}),$ Theorem \ref{th1} addresses
only the integrability of $|x|^{c}|u|^{r}$ when $u$ is viewed as a
distribution on $\Bbb{R}_{*}^{N}.$ For instance, if $u=\delta $ (Dirac
delta), then $\delta $ is the $0$ function on $\Bbb{R}_{*}^{N}$ and Theorem 
\ref{th1} with $c=0$ and $r=1$ implies the trivial $0\in L^{1}(\Bbb{R}^{N})$
but not the absurd $\delta \in L^{1}(\Bbb{R}^{N}).$ It is only when $u\in
L_{loc}^{1}(\Bbb{R}^{N})$ that the integrability of $|x|^{c}|u|^{r}$ is
independent of whether $u$ is viewed as a distribution on $\Bbb{R}^{N}$ or $%
\Bbb{R}_{*}^{N}.$

The necessity of the conditions given in Theorem \ref{th1} is proved in the
next section, where it is also shown that the inequality (\ref{8}) in part
(v) holds if the sufficiency of $a\neq 0$ together with (i), (ii) or (iii)
is assumed (Corollary \ref{cor3}).

Since the norm (\ref{3}) incorporates a supremum norm, a classical two-step
approach to the sufficiency, first for some subclass of functions with
bounded supports, followed by a denseness argument, is clearly hopeless in
general\footnote{%
It also fails when $q<\infty ,$ though for less obvious reasons; see \cite
{Ra11}.}. For example, if $b-p<ap-N$ and $u$ is a smooth function on $\Bbb{R}%
^{N}$ vanishing on a neighborhood of $0$ and equal to $|x|^{-a}$ for large $%
|x|,$ then $u\in $ $W_{\{a,b\}}^{1,(\infty ,p)}(\Bbb{R}_{*}^{N}),$ but $u$
cannot be approximated by functions of $W_{\{a,b\}}^{1,(\infty ,p)}(\Bbb{R}%
_{*}^{N})$ with bounded support since $||\,|x|^{a}(u-v)||_{\infty }\geq 1$
for every such function $v.$

The sufficiency of $a\neq 0$ plus one of the conditions (i) to (iv) is
proved in four steps, after the short Section \ref{background} of background
material. The case $r=p$ is resolved first (Section \ref{r=p}, Theorem \ref
{th8}) and used to handle $1\leq r<p$ (Section \ref{1<r<p}, Theorem \ref
{th10}) and $p<r\leq p^{*}$ (Section \ref{p<r<p*}, Theorem \ref{th13}). All
three proofs also rely upon various special cases of the main embedding
theorem in \cite{Ra11}. Lastly, when $1\leq p<N$ and $r>p^{*},$ the
embedding is deduced from the case $r=p^{*}$ through a nonlinear ``change of
variable''. The (necessary) restriction $\theta _{c}\leq \frac{p^{*}}{r}$ is
crucial to the success of this procedure (Section \ref{r>p*}, Theorem \ref
{th16}). The multiplicative inequalities (\ref{9}) and (\ref{10}) are proved
in Theorems \ref{th8}, \ref{th13} and \ref{th16}.

When $1\leq r\leq p$ (Sections \ref{r=p} and \ref{1<r<p}), the embedding
theorem is actually stronger than Theorem \ref{th1} since it proves the
embedding into $L^{r}(\Bbb{R}^{N};|x|^{c}dx)$ of the \emph{larger} space 
\begin{multline}
\widetilde{W}_{\{a,b\}}^{1,(\infty ,p)}:=  \label{11} \\
\{u\in L_{loc}^{1}(\Bbb{R}_{*}^{N}):|x|^{a}u\in L^{\infty }(\Bbb{R}%
^{N}),\quad \partial _{\rho }u\in L^{p}(\Bbb{R}^{N};|x|^{b}dx)\},
\end{multline}
with the \emph{weaker} norm\footnote{%
We shall not need a notation for the norm of $W_{\{a,b\}}^{1,(\infty ,p)}(%
\Bbb{R}_{*}^{N})$ in (\ref{3}).} 
\begin{equation}
||u||_{\{a,b\},(\infty ,p)}:=||\,|x|^{a}u||_{\infty }+||\partial _{\rho
}u||_{b,p},  \label{12}
\end{equation}
where, in (\ref{11}) and (\ref{12}), $\partial _{\rho }u=\nabla u\cdot \frac{%
x}{|x|}$ denotes the radial derivative of $u$ (well defined for every
distribution on $\Bbb{R}_{*}^{N}).$ Thus, the embedding requires no
integrability assumption about the first derivatives, except for the radial
one. This is no longer true when $r>p,$ when the embedding is only proved
for the space $W_{\{a,b\}}^{1,(\infty ,p)}(\Bbb{R}_{*}^{N}).$

The multiplicative inequalities (\ref{8}), (\ref{9}) and (\ref{10}) are
extensions of the CKN inequalities \cite{CaKoNi84} since, in addition to
requiring $u\in C_{0}^{\infty }(\Bbb{R}^{N}),$ the latter do not incorporate
supremum norms. The same thing can be said of related inequalities of Maz'ya 
\cite[Theorem 9]{Ma73}, \cite[p.127]{Ma85} (see also the expanded text \cite
{Ma11}), more general but less explicit than the CKN inequalities.

As a more concrete example, the real numbers $a,b$ and $1\leq p,r<\infty $
such that $W_{\{a,b\}}^{1,(\infty ,p)}(\Bbb{R}_{*}^{N})$ is continuously
embedded in the unweighted space $L^{r}(\Bbb{R}^{N})$ (so that $c=0$) are
characterized in Section \ref{example} (Theorem \ref{th17}). When $N=1,r\geq
p$ and $b=p\left( 1+\frac{1}{r}-\frac{1}{p}\right) $ (and $a>0,c=0$) we show
that Theorem \ref{th17} can be deduced from a weighted Hardy-type inequality
of Bradley \cite{Br78}, or even by a simple integration by parts if $%
a=b=p=r=1.$ We also show that if $b=c=0$ (Corollary \ref{cor18}), another
proof can be derived from Sobolev's inequality irrespective of $N.$

Even though the examples of Section \ref{example} show that there is no
doubt that the inequalities of this paper must be known for some values of
the parameters, no systematic investigation seems to be on record, even for $%
C_{0}^{\infty }(\Bbb{R}^{N})$ or $C_{0}^{\infty }(\Bbb{R}_{*}^{N})$ and/or $%
N=1.$ On the other hand, for functions of $C_{0}^{\infty }(\Bbb{R}_{*}^{N}),$
the multiplicative inequalities, including those of \cite{Ra11}, hold for a
wider range of parameters than stipulated in Theorem \ref{th1}. No proof of
this claim will be given here, but when $p=q=r=2$ and $c=\frac{a+b}{2}-1,$ a
result of this type was recently obtained by Catrina and Costa \cite{CaCo09}.

The last section of \cite{Ra11} explains, in broad terms, how the embedding
theorem of that paper can be used to prove more general ones when the
weights have power-like singularities at a finite number of points and at
infinity. The interested reader should have no difficulty to see how Theorem 
\ref{th1} above fits into that discussion.

\begin{remark}
\label{rm1} Up to and including Section \ref{r=p}, the following will be
used repeatedly: The Kelvin transform $x\in \Bbb{R}_{*}^{N}\mapsto
x|x|^{-2}\in \Bbb{R}_{*}^{N}$ induces an isometry from $\widetilde{W}
_{\{a,b\}}^{1,(\infty ,p)}$ ($W_{\{a,b\}}^{1,(\infty ,p)}(\Bbb{R}_{*}^{N})$)
onto $\widetilde{W}_{\{-a,2p-2N-b\}}^{1,(\infty ,p)}$ ($W_{\{-a,2p-2N-b%
\}}^{1,(\infty ,p)}(\Bbb{R}_{*}^{N})$) and from $L^{r}(\Bbb{R}%
^{N};|x|^{c}dx) $ onto $L^{r}(\Bbb{R}^{N};|x|^{-2N-c}dx).$ In practice, this
will be helpful to shorten proofs when two sets of assumptions about $a$ and 
$b$ are exchanged into one another by Kelvin transform.
\end{remark}

Everywhere in the paper, $C>0$ denotes a constant whose value may not be the
same in different places. Also, $\zeta \in C_{0}^{\infty }(\Bbb{R}^{N})$ is
chosen once and for all such that $0\leq \zeta \leq 1,$ $\zeta (x)=1$ if $%
|x|\leq \frac{1}{2}$ and $\zeta (x)=0$ if $|x|\geq 1.$

\section{Necessity\label{necessary}}

In this section, we prove that the hypotheses of Theorem \ref{th1} are
necessary in the more general case when $r>0$ and $N$ is arbitrary; recall $%
r\geq 1$ is assumed in Theorem \ref{th1} when $N>1.$

\begin{theorem}
\label{th2}Let $a,b,c\in \Bbb{R}$ and $1\leq p<\infty ,0<r<\infty $ be
given. Then, $W_{\{a,b\}}^{1,(\infty ,p)}(\Bbb{R}_{*}^{N})$ (hence \textit{a
fortiori} $\widetilde{W}_{\{a,b\}}^{1,(\infty ,p)}$) is not contained $L^{r}(%
\Bbb{R}^{N};|x|^{c}dx)$ if one of the following conditions holds:\newline
(i) $a=0.$\newline
(ii) $c$ does not belong to the closed interval with endpoints $c^{0}$ and $%
c^{1}.$\newline
(iii) $ap-N\neq b-p,$ $c=c^{0}$.\newline
(iv) $b-p\leq -N,a>0$ or $b-p\geq -N,a<0$ and $c$ does not belong to the
open interval with endpoints $c^{0}$ and $-N.$\newline
Furthermore, $W_{\{a,b\}}^{1,(\infty ,p)}(\Bbb{R}_{*}^{N})$ (hence \textit{a
fortiori} $\widetilde{W}_{\{a,b\}}^{1,(\infty ,p)}$) is not continuously%
\footnote{%
In principle, this does not rule out $W_{\{a,b\}}^{1,(\infty ,p)}(\Bbb{R}
_{*}^{N})\subset L^{r}(\Bbb{R}^{N};|x|^{c}dx).$} embedded into $L^{r}(\Bbb{R}%
^{N};|x|^{c}dx)$ if:\newline
(v) $c=c^{1}$ and $r<p.$\newline
\end{theorem}

\begin{proof}
(i) $u=1$ provides a counterexample.

(ii) If $c<\min \left\{ c^{0},c^{1}\right\} ,$ let $u(x):=|x|^{-\frac{c+N}{r}
}\zeta (x).$ Then, $u\notin L^{r}(\Bbb{R}^{N};|x|^{c}dx)$ since $%
|x|^{c}|u(x)|^{r}=|x|^{-N}$ on a neighborhood of $0,$ but $u\in
W_{\{a,b\}}^{1,(\infty ,p)}(\Bbb{R}_{*}^{N})$ since $a-\frac{c+N}{r}>0,b-p-%
\frac{p(c+N)}{r}>-N,\zeta $ has compact support and $\nabla \zeta $ has
compact support and vanishes on a neighborhood of $0.$

If $c>\max \{c^{0},c^{1}\},$ let $u(x):=|x|^{-\frac{c+N}{r}}(1-\zeta (x))$
and argue as above, with obvious modifications.

(iii) If $ap-N\neq b-p,$ then $c^{1}\neq c^{0}$ and the argument of (ii)
continues to work when $c=c^{0}.$

(iv) By Kelvin transform (Remark \ref{rm1}), it suffices to consider $%
b-p\leq -N$ and $a>0.$ If so, $ap-N>b-p$ and $c^{1}\leq -N<c^{0},$ so that $%
W_{\{a,b\}}^{1,(\infty ,p)}(\Bbb{R}_{*}^{N})\nsubseteq L^{r}(\Bbb{R}
^{N};|x|^{c}dx)$ if $c\geq c^{0}$ by (ii) and (iii). If now $c\leq -N,$ then 
$\zeta \notin L^{r}(\Bbb{R}^{N};|x|^{c}dx)$ since $\zeta =1$ on a
neighborhood of $0$ but $\zeta \in W_{\{a,b\}}^{1,(\infty ,p)}(\Bbb{R}%
_{*}^{N})$ since $a>0,\zeta $ has compact support and $\nabla \zeta $ has
compact support and vanishes on a neighborhood of $0.$

(v) The argument is different when $ap-N\neq b-p$ and when $ap-N=b-p.$

\textit{Case (v-{\scriptsize 1}):} $ap-N\neq b-p.$

By Kelvin transform and part (i), it suffices to consider the case when $%
a<0. $ By contradiction, if $W_{\{a,b\}}^{1,(\infty ,p)}(\Bbb{R}
_{*}^{N})\hookrightarrow L^{r}(\Bbb{R}^{N};|x|^{c^{1}}dx),$ then $%
||u||_{c^{1},r}\leq C(||\,|x|^{a}u||_{\infty }+||\nabla u||_{b,p})$ for
every $u\in W_{\{a,b\}}^{1,(\infty ,p)}(\Bbb{R}_{*}^{N}).$ Upon replacing $%
u(x)$ by $u(\lambda x)$ with $\lambda >0,$ this yields $||u||_{c^{1},r}\leq
C(\lambda ^{k}||\,|x|^{a}u||_{\infty }+||\nabla u||_{b,p})$ where $k:=\frac{%
b-p+N}{p}-a\neq 0$ (this uses $c^{1}=\frac{r(b-p+N)}{p}-N$). Thus, $%
||u||_{c^{1},r}\leq C||\nabla u||_{b,p}$ by letting $\lambda $ tend to $0$
or $\infty .$ In particular, this holds when $u(x)=f(|x|)$ with $f\in
W_{loc}^{1,p}(0,\infty ),f\geq 0,f=0$ on a neighborhood of $0$ and $f=M$
(constant) on a neighborhood of $\infty $ (if so, $u(x)=f(|x|)$ is in $%
W_{\{a,b\}}^{1,(\infty ,p)}(\Bbb{R}_{*}^{N})$ irrespective of $a<0,b\in \Bbb{%
R}$ and $p\geq 1$) and so $||f||_{c^{1}+N-1,r}\leq C||f^{\prime
}||_{b+N-1,p} $ for every such $f.$ That it is not so when $0<r<p$ is shown
in the proof of \cite[Theorem 2.1 (iv)]{Ra11}.

\textit{Case (v-{\scriptsize 2}):} $ap-N=b-p.$

If so, $c^{1}=ar-N$ ($=c^{0}$) and a direct rescaling as above is
inoperative. By contradiction, if $||u||_{c^{1},r}\leq
C(||\,|x|^{a}u||_{\infty }+||\nabla u||_{b,p})$ for every $u\in
W_{\{a,b\}}^{1,(\infty ,p)}(\Bbb{R}_{*}^{N}),$ this inequality holds when $%
u(x)=f(|x|)$ with $f\in C_{0}^{\infty }(0,\infty )$ and then $%
||f||_{ar-1,r}\leq C(||t^{a}f||_{\infty }+||f^{\prime }||_{ap+p-1,p}),$
where $b=ap+p-N$ was used. Every such $f$ has the form $f(t)=t^{-a}g(\ln t)$
with $g\in C_{0}^{\infty }(\Bbb{R}),$ whence $||g||_{r}\leq C(||g||_{\infty
}+||g||_{p}+||g^{\prime }||_{p})$ (unweighted inequality) by the change of
variable $\ln t=s.$

By choosing $g\neq 0$ and replacing $g(s)$ by $g(\lambda s)$ with $\lambda
>0,$ it follows that $\lambda ^{-\frac{1}{r}}I_{1}\leq C(I_{2}+\lambda ^{-%
\frac{1}{p}}I_{3}+\lambda ^{\frac{1}{p^{\prime }}}I_{4}),$ where $%
I_{1},...,I_{4}>0$ are independent of $\lambda >0.$ This requires $\frac{1}{r%
}\leq \frac{1}{p}$ and so $r\geq p.$ In other words, the embedding cannot be
continuous if $r<p.$
\end{proof}

As a corollary, we find that when $ap-N\neq b-p,$ the embedding is
characterized by a multiplicative, rather than just additive, norm
inequality:

\begin{corollary}
\label{cor3}Let $a,b,c\in \Bbb{R}$ and $1\leq p<\infty ,0<r<\infty ,$ be
such that $ap-N\neq b-p.$ Then, $W_{\{a,b\}}^{1,(\infty ,p)}(\Bbb{R}%
_{*}^{N})\hookrightarrow L^{r}(\Bbb{R}^{N};|x|^{c}dx)$ if and only if $a\neq
0,c$ is in the closed interval with endpoints $c^{0}$ and $c^{1}$ and there
is $C>0$ such that 
\begin{equation}
||u||_{c,r}\leq C||\nabla u||_{b,p}^{\theta _{c}}||\,|x|^{a}u||_{\infty
}^{1-\theta _{c}},\qquad \forall u\in W_{\{a,b\}}^{1,(\infty ,p)}(\Bbb{R}%
_{*}^{N}),  \label{13}
\end{equation}
where $\theta _{c}$ is given by (\ref{4}) and (\ref{5}). The same property
is true upon replacing $W_{\{a,b\}}^{1,(\infty ,p)}(\Bbb{R}_{*}^{N})$ by $%
\widetilde{W}_{\{a,b\}}^{1,(\infty ,p)}$ and (\ref{13}) by 
\begin{equation}
||u||_{c,r}\leq C||\partial _{\rho }u||_{b,p}^{\theta
_{c}}||\,|x|^{a}u||_{\infty }^{1-\theta _{c}},\qquad \forall u\in \widetilde{%
W}_{\{a,b\}}^{1,(\infty ,p)}.  \label{14}
\end{equation}
\end{corollary}

\begin{proof}
In both cases, the sufficiency follows from the generalized
arithmetic-geometric inequality. We prove the necessity for $\widetilde{W}
_{\{a,b\}}^{1,(\infty ,p)}.$ Similar arguments work in the case of $%
W_{\{a,b\}}^{1,(\infty ,p)}(\Bbb{R}_{*}^{N}).$

That $a\neq 0$ is necessary was shown in part (i) of Theorem \ref{th2}.
Suppose then $a\neq 0$ and $W_{\{a,b\}}^{1,(\infty ,p)}(\Bbb{R}%
_{*}^{N})\hookrightarrow L^{r}(\Bbb{R}^{N};|x|^{c}dx).$ By part (ii) of
Theorem \ref{th1}, $c$ is in the closed interval with (distinct) endpoints $%
c^{0}$ and $c^{1}.$ Furthermore, $||u||_{c,r}\leq C(||\,|x|^{a}u||_{\infty
}+||\partial _{\rho }u||_{b,p})$ for every $u\in \widetilde{W}%
_{\{a,b\}}^{1,(\infty ,p)}.$ In this inequality, replace $u(x)$ by $%
u(\lambda x)$ with $\lambda >0$ to get 
\begin{multline}
||u||_{c,r}\leq C\lambda ^{\frac{c+N}{r}-a}||\,|x|^{a}u||_{\infty }+C\lambda
^{\frac{c+N}{r}-\frac{b-p+N}{p}}||\partial _{\rho }u||_{b,p}=  \label{15} \\
C\lambda ^{\theta _{c}\frac{c^{1}-c^{0}}{r}}||\,|x|^{a}u||_{\infty
}+C\lambda ^{(1-\theta _{c})\frac{c^{0}-c^{1}}{r}}||\partial _{\rho
}u||_{b,p}.
\end{multline}

If $c=c^{0}$ ($c=c^{1}$), then $\theta _{c}=0$ ($\theta _{c}=1$), so that $%
||u||_{c,r}\leq C||\,|x|^{a}u||_{\infty }$ ($||u||_{c,r}\leq C||\partial
_{\rho }u||_{b,p}$), i.e., (\ref{14}) holds, by letting $\lambda $ tend to $%
0 $ or to $\infty .$ Otherwise, (\ref{14}) follows by minimizing the
right-hand side of (\ref{15}) for $\lambda >0.$ This changes $C,$ which
however remains independent of $u$ even though the minimizer is of course $u$%
-dependent. In that regard, observe that if $\theta _{c}>0,$ it follows from
(\ref{15}) that $u=0$ if $\partial _{\rho }u=0,$ once again by letting $%
\lambda $ tend to $0$ or to $\infty .$ Thus, it is not restrictive to assume 
$||\,|x|^{a}u||_{\infty }>0$ and $||\partial _{\rho }u||_{b,p}>0$ in the
minimization step.
\end{proof}

The next corollary gives an additional necessary condition for the
continuity of the embedding when $r>p^{*}.$

\begin{corollary}
\label{cor4}Let $a,b,c\in \Bbb{R}$ and $1\leq p<N,p^{*}<r<\infty $ be given.
If $ap-N\neq b-p$ and $W_{\{a,b\}}^{1,(\infty ,p)}(\Bbb{R}%
_{*}^{N})\hookrightarrow L^{r}(\Bbb{R}^{N};|x|^{c}dx),$ then $\theta
_{c}\leq \frac{p^{*}}{r}<1.$
\end{corollary}

\begin{proof}
First, $\theta _{c}\in [0,1]$ (even $(0,1]$) by Theorem \ref{th2},
irrespective of $p$ and $r.$ Next, let $\varphi \in C_{0}^{\infty }(\Bbb{R}%
^{N}),\varphi \neq 0,$ be chosen once and for all. If $x_{0}\in \Bbb{R}^{N}$
and $R:=|x_{0}|$ is large enough, then $\varphi (\cdot +x_{0})\in
C_{0}^{\infty }(\Bbb{R}_{*}^{N})\subset W_{\{a,b\}}^{1,(\infty ,p)}(\Bbb{R}%
_{*}^{N}) $ irrespective of $a,b$ and $p.$ By using (\ref{13}) with $%
u=\varphi (\cdot +x_{0})$ and by letting $R\rightarrow \infty ,$ we get
(because $\limfunc{Supp}\varphi $ is compact) $R^{\frac{c}{r}}||\varphi
||_{r}\leq CR^{\frac{ b\theta _{c}}{p}+a(1-\theta _{c})}||\nabla \varphi
||_{p}^{\theta _{c}}||\,\varphi ||_{\infty }^{1-\theta _{c}}$ with $C>0$
independent of $R>0 $ large enough. This implies $\frac{c}{r}\leq \frac{
b\theta _{c}}{p}+a(1-\theta _{c}).$ By adding $\frac{N}{r}$ to both sides
and using (\ref{7}), it follows that $\theta _{c}\left( \frac{1}{p}-\frac{1}{
N}\right) \leq \frac{1}{r}.$ This is always true if $p\geq N$ or if $p<N$
and $r\leq p^{*},$ but is equivalent to $\theta _{c}\leq \frac{p^{*}}{r}$($%
<1 $) if $p<N$ and $r>p^{*}.$
\end{proof}

It is a simple matter to check that, together, Theorem \ref{th2} and
Corollaries \ref{cor3} and \ref{cor4} imply that the hypotheses of Theorem 
\ref{th1} are necessary.

\section{Background\label{background}}

In this section, we collect a few preliminary results needed at various
stages of the proof of Theorem \ref{th1}. The material in the first
subsection is mostly taken from \cite[Section 3]{Ra11}. A proof is given
only for Lemma \ref{lm6}, not used in that reference.

\subsection{\protect\smallskip The space $\widetilde{W}_{loc}^{1,1}$\label%
{space}}

If $u\in L_{loc}^{1}(\Bbb{R}_{*}^{N}),$ define the spherical mean of $u$ 
\begin{equation}
f_{u}(t):=(N\omega _{N})^{-1}\int_{\Bbb{S}^{N-1}}u(t\sigma )d\sigma ,
\label{16}
\end{equation}
where $\omega _{N}$ is the volume of the unit ball of $\Bbb{R}^{N}.$ By
Fubini's theorem in spherical coordinates, $f_{u}(t)$ is well defined for
a.e. $t>0$ and $f_{u}\in L_{loc}^{1}(0,\infty ).$ Note that $u$ is radially
symmetric if and only if $u(x)=f_{u}(|x|).$ More generally, $%
u_{S}(x):=f_{u}(|x|)$ is the \emph{radial symmetrization }of $u.$ Set 
\begin{equation*}
\widetilde{W}_{loc}^{1,1}:=\{u\in L_{loc}^{1}(\Bbb{R}_{*}^{N}):\partial
_{\rho }u\in L_{loc}^{1}(\Bbb{R}_{*}^{N})\}.
\end{equation*}
If $u\in \widetilde{W}_{loc}^{1,1},$ then $f_{u}\in W_{loc}^{1,1}(0,\infty )$
and $f_{u}^{\prime }(t):=(N\omega _{N})^{-1}\int_{\Bbb{S}^{N-1}}\partial
_{\rho }u(t\sigma )d\sigma .$ Also, $|u|\in \widetilde{W}_{loc}^{1,1}$ and $%
\partial _{\rho }|u|=(\limfunc{ sgn}u)\partial _{\rho }u$ where $\limfunc{sgn%
}u:=0$ a.e. on $u^{-1}(0).$ Thus, $f_{|u|}\in W_{loc}^{1,1}(0,\infty ),$ so
that $f_{|u|}$ is continuous on $(0,\infty )$ and the subsets 
\begin{equation}
\widetilde{W}_{loc,-}^{1,1}:=\{u\in \widetilde{W}_{loc}^{1,1}:\underline{%
\lim }_{t\rightarrow \infty }f_{|u|}(t)=0\},  \label{17}
\end{equation}
\begin{equation}
\widetilde{W}_{loc,+}^{1,1}:=\{u\in \widetilde{W}_{loc}^{1,1}:\underline{%
\lim }_{t\rightarrow 0^{+}}f_{|u|}(t)=0\},  \label{18}
\end{equation}
are well defined. A few properties needed later are spelled out in the next
two lemmas.

\begin{lemma}
\label{lm5}\emph{(\cite[Lemma 3.5]{Ra11})} If $f\in W_{loc}^{1,1}(0,\infty
),f\geq 0$ and $\underline{\lim }_{t\rightarrow 0^{+}}f(t)=0$ ($\underline{%
\lim }_{t\rightarrow \infty }f(t)=0$), then $f(t)\leq \int_{0}^{t}|f^{\prime
}(\tau )|d\tau $ ($f(t)\leq \int_{t}^{\infty }|f^{\prime }(\tau )|d\tau $)
for every $t>0.$ \newline
\end{lemma}

\begin{lemma}
\label{lm6}Let $a,b\in \Bbb{R}$ and $1\leq p<\infty $ be given. If $u\in 
\widetilde{W}_{\{a,b\}}^{1,(\infty ,p)},$ then \newline
(i) $u\in \widetilde{W}_{loc,-}^{1,1}$ if $a>0$ and $u\in \widetilde{W}%
_{loc,+}^{1,1}$ if $a<0.$\newline
(ii) $|u|\in \widetilde{W}_{\{a,b\}}^{1,(\infty ,p)}$ and $%
||\,|x|^{a}|u|\,||_{\infty }=||\,|x|^{a}u\,||_{\infty },||\partial _{\rho
}|u|\,||_{b,p}=||\partial _{\rho }u||_{b,p}.$\newline
(iii) $v:=[(|u|^{p})_{S}]^{\frac{1}{p}}\in \widetilde{W}_{\{a,b\}}^{1,(%
\infty ,p)}$ and $||\,|x|^{a}v||_{\infty }\leq ||\,|x|^{a}u||_{\infty
},||\partial _{\rho }v\,||_{b,p}\leq ||\partial _{\rho }u||_{b,p}.$
\end{lemma}

\begin{proof}
Obviously, $\widetilde{W}_{\{a,b\}}^{1,(\infty ,p)}\subset \widetilde{W}
_{loc}^{1,1}$ irrespective of $a,b$ and $p.$

(i) By (\ref{16}), $t^{a}f_{|u|}$ is bounded on $(0,\infty ).$ Thus, $%
\lim_{t\rightarrow \infty }f_{|u|}(t)=0$ if $a>0$ and $\lim_{t\rightarrow
0^{+}}f_{|u|}(t)=0$ if $a<0.$

(ii) Since $u\in \widetilde{W}_{loc}^{1,1},$ then $\partial _{\rho }|u|=(%
\limfunc{sgn}u)\partial _{\rho }u$ (as mentioned earlier). With this, the
proof is trivial.

(iii) That $|x|^{a}v\in L^{\infty }(\Bbb{R}^{N})$ and $||\,|x|^{a}v||_{%
\infty }\leq ||\,|x|^{a}u||_{\infty }$ follows from $||\,|x|^{a}v||_{\infty
}^{p}=||\,|x|^{ap}v^{p}||_{\infty }=||\,|x|^{ap}(|u|^{p})_{S}||_{\infty
}\leq ||\,|x|^{ap}|u|^{p}||_{\infty }$ (by (\ref{16}) with $u$ replaced by $%
|u|^{p}$ and $(|u|^{p})_{S}(x):=f_{|u|^{p}}(|x|)$) and from $%
||\,|x|^{ap}|u|^{p}||_{\infty }=||\,|x|^{a}u||_{\infty }^{p}.$

The proof that $\partial _{\rho }v\in L^{p}(\Bbb{R}^{N};|x|^{b}dx)$ with $%
||\partial _{\rho }v\,||_{b,p}\leq ||\partial _{\rho }u||_{b,p}$ is more
delicate, but identical to the proof given in \cite[Lemma 5.1]{Ra11} when,
with the notation of that paper, $r=p\leq q<\infty $ and $u\in \widetilde{W}%
_{\{a,b\}}^{1,(q,p)}.$
\end{proof}

\subsection{A Hardy-type inequality\label{hardy}}

If $\alpha <-1$ and $1\leq p<\infty ,$ the inequality 
\begin{equation}
\left( \int_{0}^{\infty }t^{\alpha }\left( \int_{0}^{t}g(\tau )d\tau \right)
^{p}dt\right) ^{\frac{1}{p}}\leq C\left( \int_{0}^{\infty }t^{\alpha
+p}g(t)^{p}dt\right) ^{\frac{1}{p}},  \label{19}
\end{equation}
holds for some constant $C>0$ and every measurable function $g\geq 0$ on $%
(0,\infty ).$ This is a special case of an inequality of Muckenhoupt \cite
{Mu72} for general (compatible) weights. If $\alpha =-p$ with $p>1,$ Hardy's
inequality is recovered.

\section{The embedding theorem when $r=p$\label{r=p}}

Let $d,b\in \Bbb{R}$ and $1\leq p<\infty .$ In analogy with (\ref{11}), we
define the space 
\begin{equation*}
\widetilde{W}_{\{d,b\}}^{1,(p,p)}:=\{u\in L_{loc}^{1}(\Bbb{R}_{*}^{N}):u\in
L^{p}(\Bbb{R}^{N};|x|^{d}dx),\quad \partial _{\rho }u\in L^{p}(\Bbb{R}
^{N};|x|^{b}dx)\},
\end{equation*}
with norm $||u||_{\{d,b\},(p,p)}:=||u||_{d,p}+||\partial _{\rho }u||_{b,p}.$

The next lemma is a special case of \cite[Theorem 5.2]{Ra11}.

\begin{lemma}
\label{lm7}Let $\,b,c,d\in \Bbb{R}$ and $1\leq p<\infty $ be given. Then, $%
\widetilde{W}_{\{d,b\}}^{1,(p,p)}\hookrightarrow L^{p}(\Bbb{R}%
^{N};|x|^{c}dx) $ (and hence $\widetilde{W}_{\{d,b\}}^{1,(p,p)}%
\hookrightarrow \widetilde{W}_{\{c,b\}}^{1,(p,p)}$) if one of the following
conditions holds: \newline
(i) $d$ and $b-p$ are on the same side of $-N$ (including $-N$), $d\neq b-p$
and $c$ is in the semi-open interval with endpoints $d$ (included) and $b-p$
(not included).\newline
(ii) $d$ and $b-p$ are strictly on opposite sides of $-N$ and $c$ is in the
semi-open interval with endpoints $d$ (included) and $-N$ (not included).
\end{lemma}

The conditions (i) and (ii) of the lemma are not necessary: There is a third
option with no relevance to the issue of interest here.

\begin{theorem}
\label{th8}Let $a,b,c\in \Bbb{R}$ and $1\leq p<\infty $ be given. Then, $%
\widetilde{W}_{\{a,b\}}^{1,(\infty ,p)}\hookrightarrow L^{p}(\Bbb{R}
^{N};|x|^{c}dx)$ if and only if $a\neq 0$ and one of the following three
conditions holds: \newline
(i) $ap-N\neq b-p$ are on the same side of $-N$ (including $b-p=-N$) and $c$
is in the open interval with endpoints $ap-N$ and $b-p.$\newline
(ii) $ap-N$ and $b-p$ are strictly on opposite sides of $-N$ and $c$ is in
the open interval with endpoints $ap-N$ and $-N.$\newline
(iii) $a(b-p+N)>0$ and\footnote{%
Note that $b-p=c^{1}$ in (\ref{3}) when $r=p$ and that it is \emph{not}
assumed that $ap-N\neq b-p.$} $c=b-p.$ If so, there is a constant $C>0$ such
that 
\begin{equation}
||u||_{b-p,p}\leq C||\partial _{\rho }u||_{b,p},\qquad \forall u\in 
\widetilde{W}_{\{a,b\}}^{1,(\infty ,p)}.  \label{20}
\end{equation}
\end{theorem}

\begin{proof}
The necessity follows from Theorem \ref{th2} with $r=p$ (hence $c^{0}=ap-N$
and $c^{1}=b-p$). To prove the sufficiency, we first choose $\zeta \in
C_{0}^{\infty }(\Bbb{R}^{N})$ as in the end of the Introduction. It is
readily checked that the multiplication by $\zeta $ or $1-\zeta $ is
continuous on $\widetilde{W}_{\{a,b\}}^{1,(\infty ,p)}$ (just notice that $%
|x|^{b-ap}$ is integrable on $\limfunc{Supp}\nabla \zeta $ irrespective of $%
a,b$ and $p$). Thus, the problem is reduced to showing that $||\zeta
u||_{c,p}\leq C||\zeta u||_{\{a,b\},(\infty ,p)}$ and $||(1-\zeta
)u||_{c,p}\leq C||(1-\zeta )u||_{\{a,b\},(\infty ,p)}$ where $C>0$ is
independent of $u.$

(i) By Kelvin transform, a proof is needed only when $ap-N>-N$ (i.e., $a>0$)
and $b-p\geq -N.$ Since $ap-N\neq b-p,$ this splits into the two cases $%
-N<ap-N<b-p$ and $-N\leq b-p<ap-N.$

\textit{Case (i-{\scriptsize 1}):}$-N<ap-N<b-p.$

Let $c\in (ap-N,b-p).$ That $||\zeta u||_{c,p}\leq C||\,|x|^{a}\zeta
u||_{\infty }\leq C||\zeta u||_{\{a,b\},(\infty ,p)}$ is simply due to $%
|x|^{c-ap} $ being integrable on $\limfunc{Supp}\zeta \subset B(0,1)$ since $%
c-ap>-N.$

Next, pick $d\in (-N,ap-N).$ Then, $||(1-\zeta )u||_{d,p}\leq
C||\,|x|^{a}(1-\zeta )u||_{\infty }$ because $d-ap<-N,$ so that $|x|^{d-ap}$
is integrable on $\limfunc{Supp}(1-\zeta )\subset \Bbb{R}^{N}\backslash B(0,%
\frac{1}{2}).$ Thus, $(1-\zeta )u\in \widetilde{W}_{\{d,b\}}^{1,(p,p)}$ and $%
||(1-\zeta )u||_{\{d,b\},(p,p)}\leq C||(1-\zeta )u||_{\{a,b\},(\infty ,p)}.$
Since $c\in (d,b-p)$ and $d,b-p>-N,$ part (i) of Lemma \ref{lm7} yields $%
\widetilde{W}_{\{d,b\}}^{1,(p,p)}\hookrightarrow L^{p}(\Bbb{R}
^{N};|x|^{c}dx) $ and so $||(1-\zeta )u||_{c,p}\leq C||(1-\zeta
)u||_{\{a,b\},(\infty ,p)}$ for another constant $C$ by compounding
inequalities.

\textit{Case (i-{\scriptsize 2}):} $-N\leq b-p<ap-N.$

If $c\in (b-p,ap-N),$ what is now obvious is that $||(1-\zeta )u||_{c,p}\leq
C||\,|x|^{a}(1-\zeta )u||_{\infty }\leq C||(1-\zeta )u||_{\{a,b\},(\infty
,p)}.$ To prove $||\zeta u||_{c,p}\leq C||\zeta u||_{\{a,b\},(\infty ,p)},$
choose $d>ap-N$ and argue as in Case (i-{\scriptsize 1}), with minor
modifications. Specifically, $||\zeta u||_{d,p}\leq C||\,|x|^{a}\zeta
u||_{\infty }$ because $|x|^{d-ap}$ is integrable on $\limfunc{Supp}\zeta ,$
so that $\zeta u\in \widetilde{W}_{\{d,b\}}^{1,(p,p)}$with $||\zeta
u||_{\{d,b\},(p,p)}\leq C||\zeta u||_{\{a,b\},(\infty ,p)},$ while $%
\widetilde{W}_{\{d,b\}}^{1,(p,p)}\hookrightarrow L^{p}(\Bbb{R}
^{N};|x|^{c}dx) $ by part (i) of Lemma \ref{lm7} since $c\in
(b-p,ap-N)\subset (b-p,d)$

(ii) By Kelvin transform, it suffices to discuss the case when $b-p<-N<ap-N.$
Let $c\in (-N,ap-N)$ be given. As in Case (i-{\scriptsize 2}) above, it is
plain that $||(1-\zeta )u||_{c,p}\leq C||\,|x|^{a}(1-\zeta )u||_{\infty
}\leq C||(1-\zeta )u||_{\{a,b\},(\infty ,p)}.$ The proof that $||\zeta
u||_{c,p}\leq C||\zeta u||_{\{a,b\},(\infty ,p)}$ proceeds as in Case (i-%
{\scriptsize 2}), by first choosing $d>ap-N$ to get $||\zeta u||_{d,p}\leq
C||\,|x|^{a}\zeta u||_{\infty },$ but next using part (ii) of Lemma \ref{lm7}
since $c\in (-N,ap-N)\subset (-N,d).$

(iii) It suffices to prove (\ref{20}). By Kelvin transform, suppose $%
a<0,c=b-p<-N$ with no loss of generality. By part (ii) of Lemma \ref{lm6},
it is also not restrictive to assume $u\geq 0$ and, by part (iii) of that
lemma, that $u$ is radially symmetric since, when $u$ is changed into $%
[(u^{p})_{S}]^{\frac{1}{p}},$ the left-hand side of (\ref{20}) is unchanged
and its right-hand side is not increased.

Now, if $u\geq 0$ is radially symmetric, then $u(x)=f_{u}(|x|)$ with $%
f_{u}\in W_{loc}^{1,1}(0,\infty ),$\linebreak $f_{u}\geq 0$ and (\ref{20})
becomes 
\begin{equation}
||f_{u}||_{b-p+N-1,p}\leq C||f_{u}^{\prime }||_{b+N-1,p}.  \label{21}
\end{equation}
By part (i) of Lemma \ref{lm6}, $u\in \widetilde{W}_{loc,+}^{1,1}$ since $%
a<0,$ so that $f_{u}(t)\leq \int_{0}^{t}|f_{u}^{\prime }(\tau )|d\tau $ by (%
\ref{18}) and Lemma \ref{lm5}. Thus, (\ref{21}) follows from the Hardy-type
inequality (\ref{19}) with $\alpha =b-p+N-1$ ($<-1$).
\end{proof}

\section{The embedding theorem when $1\leq r<p$\label{1<r<p}}

The embedding theorem when $1\leq r<p$ (Theorem \ref{th10} below) will now
be proved by combining Theorem \ref{th8} with the following special case of 
\cite[Theorem 5.2]{Ra11}.

\begin{lemma}
\label{lm9}Let $b,c,d\in \Bbb{R}$ and $1\leq r<p<\infty $ be given ($1\leq
p<\infty $ and $0<r<\infty $ if $N=1$). Then, $\widetilde{W}%
_{\{d,b\}}^{1,(p,p)}\hookrightarrow L^{r}(\Bbb{R}^{N};|x|^{c}dx)$ if one of
the following two conditions holds:\newline
(i) $d$ and $b-p$ are on the same side of $-N$ (including $-N$), $d\neq b-p,$
$c$ is in the open interval with endpoints $\frac{r(d+N)}{p}-N$ and $\frac{%
r(b-p+N)}{p}-N.$\newline
(ii) $d$ and $b-p$ are strictly on opposite sides of $-N$ and $c$ is in the
open interval with endpoints $\frac{r(d+N)}{p}-N$ and $-N.$\newline
\end{lemma}

Lemma \ref{lm9} is also true if $r=p,$ when it coincides with Lemma \ref{lm7}%
, but the exposition is clearer by keeping the two statements separate.

\begin{theorem}
\label{th10}Let $a,b,c\in \Bbb{R}$ and $1\leq r<p<\infty $ be given ($1\leq
p<\infty $ and $0<r<\infty $ if $N=1$). Then, $\widetilde{W}%
_{\{a,b\}}^{1,(\infty ,p)}\hookrightarrow L^{r}(\Bbb{R}^{N};|x|^{c}dx)$ if
and only if $a\neq 0$ and one of the following two conditions holds: \newline
(i) $ap-N\neq b-p$ are on the same side of $-N$ (including $b-p=-N$) and $c$
is in the open interval with endpoints $c^{0}$ and $c^{1}.$\newline
(ii) $ap-N$ and $b-p$ are strictly on opposite sides of $-N$ and $c$ is in
the open interval with endpoints $c^{0}$ and $-N.$\newline
\end{theorem}

\begin{proof}
The necessity follows from Theorem \ref{th2}.

(i) By part (i) of Theorem \ref{th8} with $c$ replaced by $d,$ it follows
that $\widetilde{W}_{\{a,b\}}^{1,(\infty ,p)}\hookrightarrow L^{p}(\Bbb{R}
^{N};|x|^{d}dx)$ for every $d$ in the open interval $J$ with endpoints $ap-N$
and $b-p.$ Of course, this implies $\widetilde{W}_{\{a,b\}}^{1,(\infty
,p)}\hookrightarrow \widetilde{W}_{\{d,b\}}^{1,(p,p)}$ for $d\in J.$ Since $%
J $ is open and its endpoints are on the same side of $-N,$ it is plain that
if $d\in J,$ then $d$ and $b-p$ are on the same side of $-N$ and $d\neq b-p.$
Therefore, by Lemma \ref{lm9}, $\widetilde{W}_{\{d,b\}}^{1,(p,p)}%
\hookrightarrow L^{r}(\Bbb{R}^{N};|x|^{c}dx)$ for every $d\in J$ and every $%
c $ in the open interval $I_{d}$ with endpoints $\frac{r(d+N)}{p}-N$ and $%
\frac{r(b-p+N)}{p}-N.$\newline

Altogether, this yields $\widetilde{W}_{\{a,b\}}^{1,(\infty
,p)}\hookrightarrow L^{r}(\Bbb{R}^{N};|x|^{c}dx)$ for every $c\in \cup
_{d\in J}I_{d}$ and it is obvious that this union is the open interval with
endpoints $ar-N=c^{0}$ and $\frac{r(b-p+N)}{p}-N=c^{1}.$

(ii) Proceed as in (i), but now using parts (ii) of Theorem \ref{th8} and
Lemma \ref{lm9}.
\end{proof}

\begin{remark}
\label{rm3}For the subspace of $\widetilde{W}_{\{a,b\}}^{1,(\infty ,p)}$ of
radially symmetric functions, Theorem \ref{th10} remains true if $0<r<1:$
Just use the theorem with $N=1$ after replacing $b$ and $c$ by $b+N-1$ and $%
c+N-1,$ respectively.
\end{remark}

\section{The embedding theorem when $p<r\leq p^{*}$\label{p<r<p*}}

When $r<p,$ the proof of Theorem \ref{th10} shows that $\widetilde{W}%
_{\{a,b\}}^{1,(\infty ,p)}\hookrightarrow L^{r}(\Bbb{R}^{N};|x|^{c}dx)$ if
and only if $\widetilde{W}_{\{a,b\}}^{1,(\infty ,p)}\hookrightarrow L^{p}(%
\Bbb{R}^{N};|x|^{d}dx)$ for some suitable $d\in \Bbb{R}.$ This feature is no
longer true when $r>p,$ even when $\widetilde{W}_{\{a,b\}}^{1,(\infty ,p)}$
is replaced by the smaller space $W_{\{a,b\}}^{1,(\infty ,p)}(\Bbb{R}%
_{*}^{N}).$ Accordingly, the strategy of proof will be different. We shall
need two other special cases of the embedding theorem in \cite{Ra11}. For
clarity, they are given in two separate statements. Lemma \ref{lm11} is a
rephrasing of parts (i) and (ii) of \cite[Theorem 7.1]{Ra11} when ``$%
p<r=q\leq p^{*}$'' and the inequality in Lemma \ref{lm12} below is proved in 
\cite[Theorem 10.2]{Ra11}.

\begin{lemma}
\label{lm11}Let $b,c,d\in \Bbb{R}$ and $1\leq p<r<\infty ,r\leq p^{*}$ be
given. Then, $W_{\{d,b\}}^{1,(r,p)}(\Bbb{R}_{*}^{N})\hookrightarrow L^{r}(%
\Bbb{R}^{N};|x|^{c}dx)$ if one of the following conditions holds:\newline
(i) $d$ and $b-p$ are on the same side of $-N$ (including $-N$), $\frac{d+N}{%
r}\neq \frac{b-p+N}{p},$ $c$ is in the semi-open interval with endpoints $d$
(included) and $\frac{r(b-p+N)}{p}-N$ (not included).\newline
(ii) $d$ and $b-p$ are strictly on opposite sides of $-N$ and $c$ is in the
semi-open interval with endpoints $d$ (included) and $-N$ (not included).%
\newline
\end{lemma}

Lemma \ref{lm11} remains true if $r=p,$ when it is a special case of Lemma 
\ref{lm7} since $W_{\{d,b\}}^{1,(p,p)}(\Bbb{R}_{*}^{N})\hookrightarrow 
\widetilde{W}_{\{d,b\}}^{1,(p,p)}.$

\begin{lemma}
\label{lm12}Let $b\in \Bbb{R}$ and $1\leq p<r<\infty ,r\leq p^{*}$ be given.
If $b-p\neq -N,$ then $W_{\{b-p,b\}}^{1,(p,p)}(\Bbb{R}_{*}^{N})%
\hookrightarrow L^{r}(\Bbb{R}^{N};|x|^{c^{1}}dx),$ where (as in (\ref{4})), $%
c^{1}=\frac{r(b-p+N)}{p}-N.$ Furthermore, there is a constant $C>0$ such
that 
\begin{equation*}
||u||_{c^{1},r}\leq C||\nabla u||_{b,p},\qquad \forall u\in
W_{\{b-p,b\}}^{1,(p,p)}(\Bbb{R}_{*}^{N}).
\end{equation*}
\end{lemma}

The conditions given in Lemmas \ref{lm11} and \ref{lm12} are not necessary,
but they will suffice for our purposes.

\begin{theorem}
\label{th13}Let $a,b,c\in \Bbb{R}$ and $1\leq p<r<\infty ,r\leq p^{*},$ be
given. Then, $W_{\{a,b\}}^{1,(\infty ,p)}(\Bbb{R}_{*}^{N})\hookrightarrow
L^{r}(\Bbb{R}^{N};|x|^{c}dx)$ if and only if $a\neq 0$ and one of the
following three conditions holds: \newline
(i) $ap-N\neq b-p$ are on the same side of $-N$ (including $b-p=-N$) and $c$
is in the open interval with endpoints $c^{0}$ and $c^{1}.$\newline
(ii) $ap-N$ and $b-p$ are strictly on opposite sides of $-N$ and $c$ is in
the open interval with endpoints $c^{0}$ and $-N.$\newline
(iii) $a(b-p+N)>0$ and $c=c^{1}.$ If so, there is a constant $C>0$ such that 
\begin{equation}
||u||_{c^{1},r}\leq C||\nabla u||_{b,p},\qquad \forall u\in
W_{\{a,b\}}^{1,(\infty ,p)}(\Bbb{R}_{*}^{N}).  \label{22}
\end{equation}
\end{theorem}

\begin{proof}
Once again, the necessity follows from Theorem \ref{th2} and we only address
the sufficiency.

(i) When $d\in \Bbb{R}$ runs over the open interval with endpoints $%
c^{0}=ar-N$ and $b-p+a(r-p)=\frac{p}{r}c^{1}+(1-\frac{p}{r})c^{0},$ the
point $c_{d}:=d-a(r-p)$ runs over the open interval with endpoints $ap-N$
and $b-p.$

Note that since $0<\frac{p}{r}<1,$ the point $d$ above lies in the open
interval with endpoints $c^{0}$ and $c^{1}.$ Since both $ap-N$ and $b-p$ are
on the left (or right) of $-N,$ then both $c^{0}$ and $c^{1}$ are on the
left (or right) of $-N,$ so that $d$ and $b-p$ are always on the same side
of $-N.$

Furthermore, if $d$ is close enough to $c^{0},$ then $\frac{d+N}{r}\neq 
\frac{b-p+N}{p}$ since this holds when $d=c^{0}$ (recall $ap-N\neq b-p$).
This assumption is retained in the subsequent considerations.

By part (i) of Theorem \ref{th8} with $c$ replaced by $c_{d}:=d-a(r-p),$ it
follows that $||u||_{c_{d},p}\leq C||u||_{\{a,b\},(\infty ,p)}$ when $u\in 
\widetilde{W}_{\{a,b\}}^{1,(\infty ,p)}.$ If so, $|x|^{a}u\in L^{\infty }(%
\Bbb{R}^{N})$ and $r>p$ yield $|x|^{a(r-p)}|u|^{r-p}\in L^{\infty }(\Bbb{R}%
^{N})$ with $||\,|x|^{a(r-p)}|u|^{r-p}||_{\infty
}=||\,|x|^{a}|u|\,||_{\infty }^{r-p}.$ As a result, $%
|x|^{d}|u|^{r}=(|x|^{a(r-p)}|u|^{r-p})(|x|^{c_{d}}|u|^{p})\in L^{1}(\Bbb{R}%
^{N})$ and $||u||_{d,r}^{r}=||\,|x|^{d}|u|^{r}||_{1}\leq
||\,|x|^{a}|u|\,||_{\infty
}^{r-p}||\,|x|^{c_{d}}|u|^{p}||_{1}=||\,|x|^{a}|u|\,||_{\infty
}^{r-p}||u||_{c_{d},p}^{p}\leq C^{p}||u||_{\{a,b\},(\infty ,p)}^{r}.$ This
shows that $\widetilde{W}_{\{a,b\}}^{1,(\infty ,p)}\hookrightarrow L^{r}(%
\Bbb{R}^{N};|x|^{d}dx).$ In particular, $W_{\{a,b\}}^{1,(\infty ,p)}(\Bbb{R}%
_{*}^{N})\hookrightarrow L^{r}(\Bbb{R}^{N};|x|^{d}dx)$ and so $%
W_{\{a,b\}}^{1,(\infty ,p)}(\Bbb{R}_{*}^{N})\hookrightarrow
W_{\{d,b\}}^{1,(r,p)}(\Bbb{R}_{*}^{N}).$

As shown earlier, $d$ and $b-p$ are on the same side of $-N$ and $\frac{d+N}{
r}\neq \frac{b-p+N}{p}.$ Therefore, by part (i) of Lemma \ref{lm11}, if $c$
is in the semi-open interval with endpoints $d$ (included) and $\frac{%
r(b-p+N)}{p}-N=c^{1}$(not included), then $W_{\{d,b\}}^{1,(r,p)}(\Bbb{R}%
_{*}^{N})\hookrightarrow L^{r}(\Bbb{R}^{N};|x|^{c}dx)$ and so, from the
above, $W_{\{a,b\}}^{1,(\infty ,p)}(\Bbb{R}_{*}^{N})\hookrightarrow L^{r}(%
\Bbb{R}^{N};|x|^{c}dx).$ Since $d$ can be chosen arbitrarily close to $%
c^{0}, $ this embedding does hold for every $c$ in the open interval with
endpoints $c^{0}$ and $c^{1},$ as claimed.

(ii) Proceed as above, but now using the parts (ii) of Theorem \ref{th8} and
of Lemma \ref{lm11} (note that $ap-N$ and $c^{0}=ar-N$ are always on the
same side of $-N,$ so that $d$ and $b-p$ as well as $c_{d}$ and $b-p$ are on
opposite sides of $-N\;$if $d$ is close to $c^{0}$).

(iii) First, $W_{\{a,b\}}^{1,(\infty ,p)}(\Bbb{R}_{*}^{N})\hookrightarrow 
\widetilde{W}_{\{a,b\}}^{1,(\infty ,p)}\hookrightarrow L^{p}(\Bbb{R}
^{N};|x|^{b-p}dx),$ the latter by part (iii) of Theorem \ref{th8}, so that $%
W_{\{a,b\}}^{1,(\infty ,p)}(\Bbb{R}_{*}^{N})\hookrightarrow
W_{\{b-p,b\}}^{1,(p,p)}(\Bbb{R}_{*}^{N}).$ Next, by Lemma \ref{lm12}, $%
W_{\{b-p,b\}}^{1,(p,p)}(\Bbb{R}_{*}^{N})\hookrightarrow L^{r}(\Bbb{R}
^{N};|x|^{c^{1}}dx)$ and (\ref{22}) holds since $b-p\neq -N$ and $c^{1}=%
\frac{r(b-p+N)}{p}-N.$
\end{proof}

\begin{remark}
Set $a_{r}:=c^{1}=\frac{r(b-p+N)}{p}-N$ to make explicit the $r$-dependence.
By part (iii) of Theorem \ref{th13}, $W_{\{a,b\}}^{1,(\infty ,p)}(\Bbb{R}%
_{*}^{N})\hookrightarrow W_{\{a_{r},b\}}^{1,(r,p)}(\Bbb{R}_{*}^{N})$ if $%
a(b-p+N)>0.$ As pointed out in the Introduction of \cite{Ra11}, the space $%
W_{\{a_{r},b\}}^{1,(r,p)}(\Bbb{R}_{*}^{N})$ is actually independent of $%
p\leq r\leq p^{*},r<\infty .$ Although this will be proved elsewhere (\cite
{Ra11b}), it seems of interest to report that if $N=1$ or $p>N>1$ (hence $%
p^{*}=\infty $), this space also coincides with $W_{\{a,b\}}^{1,(\infty ,p)}(%
\Bbb{R}_{*}^{N})$ when $a=\frac{b-p+N}{p}.$ In other words, if $N=1$ or $%
p>N>1$ and if $b-p+N\neq 0,$ then $W_{\{\frac{b-p+N}{p},b\}}^{1,(\infty ,p)}(%
\Bbb{R}_{*}^{N})=W_{\{a_{r},b\}}^{1,(r,p)}(\Bbb{R}_{*}^{N})$ for every $%
p\leq r<\infty ,$ with equivalent norms.
\end{remark}

\section{The embedding theorem when $p<N$ and $r>p^{*}$\label{r>p*}}

When $p<N$ and $r>p^{*},$ the embedding theorem will be deduced from the
case $r=p^{*}$ in Theorem \ref{th13} after changing $u\in
W_{\{a,b\}}^{1,(\infty ,p)}(\Bbb{R}_{*}^{N})$ into $|u|^{\frac{r}{p^{*}}}.$
The details of the procedure follow.

\begin{lemma}
\label{lm14}Let $a,b\in \Bbb{R}$ and $1\leq p<N,p^{*}<r<\infty ,$ be given.
If $u\in W_{\{a,b\}}^{1,(\infty ,p)}(\Bbb{R}_{*}^{N}),$ then $|u|^{\frac{r}{%
p^{*}}}\in W_{\{a^{*},b^{*}\}}^{1,(\infty ,p)}(\Bbb{R}_{*}^{N})$ where 
\begin{equation}
a^{*}:=\frac{ar}{p^{*}}\text{\quad and\quad }b^{*}:=b+ap\left( \frac{r}{p^{*}%
}-1\right)  \label{23}
\end{equation}
and $||\,|x|^{a^{*}}|u|^{\frac{r}{p^{*}}}||_{\infty }=||\,|x|^{a}u||_{\infty
}^{\frac{r}{p^{*}}},$\quad $||\nabla |u|^{\frac{r}{p^{*}}}||_{b^{*},p}\leq 
\frac{r}{p^{*}}||\nabla u||_{b.p}||\,|x|^{a}u||_{\infty }^{\frac{r}{p^{*}}%
-1}.$
\end{lemma}

\begin{proof}
Since $u\in W_{\{a,b\}}^{1,(\infty ,p)}(\Bbb{R}_{*}^{N})\subset
L_{loc}^{\infty }(\Bbb{R}_{*}^{N}),$ it is clear that $|u|^{\frac{r}{p^{*}}
}\in L_{loc}^{1}(\Bbb{R}_{*}^{N}).$ Thus, everything is a routine
verification if it is shown that $\nabla \left( |u|^{\frac{r}{p^{*}}}\right) 
$ (as a distribution on $\Bbb{R}_{*}^{N}$) is $\frac{r}{p^{*}}|u|^{\frac{r}{%
p^{*}}-1}\nabla u.$ This follows from \cite[Theorem 2.1 and Remark 2.1]
{MaMi72} since $|t|^{\frac{r}{p^{*}}}$ is locally Lipschitz continuous, $%
W_{\{a,b\}}^{1,(\infty ,p)}(\Bbb{R}_{*}^{N})\subset W_{loc}^{1,1}(\Bbb{R}%
_{*}^{N})$ and $\frac{r}{p^{*}}|u|^{\frac{r}{p^{*}}-1}\nabla u\in
L_{loc}^{p}(\Bbb{R}_{*}^{N})\subset L_{loc}^{1}(\Bbb{R}_{*}^{N}).$
\end{proof}

The next Lemma is just a special case of Theorem \ref{th13}$.$

\begin{lemma}
\label{lm15}Let $a,b,c\in \Bbb{R}$ and $1\leq p<N,p^{*}<r<\infty $ be given.
If $a^{*}$ and $b^{*}$ are defined by (\ref{23}), then $W_{\{a^{*},b^{*}%
\}}^{1,(\infty ,p)}(\Bbb{R}_{*}^{N})\hookrightarrow L^{p^{*}}(\Bbb{R}%
^{N};|x|^{c}dx)$ if and only if $a\neq 0$ and one of the following three
conditions holds: \newline
(i) $a^{*}p-N\neq b^{*}-p$ are on the same side of $-N$ (including $%
b^{*}-p=-N$) and $c$ is in the open interval with endpoints $c^{0}=ar-N$ and 
$c^{*1}:=\frac{p^{*}(b^{*}-p+N)}{p}-N=\frac{p^{*}}{r}c^{1}+\left( 1-\frac{%
p^{*}}{r}\right) c^{0}.$\newline
(ii) $a^{*}p-N$ and $b^{*}-p$ are strictly on opposite sides of $-N$ and $c$
is in the open interval with endpoints $c^{0}$ and $-N.$\newline
(iii) $a(b^{*}-p+N)>0$ and $c=c^{*1}.$ Furthermore, there is a constant $C>0$
such that 
\begin{equation}
||u||_{c^{*1},p^{*}}\leq C||\nabla u||_{b^{*}.p}\qquad \forall u\in
W_{\{a^{*},b^{*}\}}^{1,(\infty ,p)}(\Bbb{R}_{*}^{N}).  \label{24}
\end{equation}
\end{lemma}

\begin{proof}
Observe that $c^{0}:=ar-N=a^{*}p^{*}-N$ and use Theorem \ref{th13} with $%
a,b,r$ replaced by $a^{*},b^{*},p^{*},$ respectively.
\end{proof}

Recall the definition of $\theta _{c}$ in (\ref{4}) and (\ref{5}) when $c$
is in the closed interval with endpoints $c^{0}$ and $c^{1}\neq c^{0}.$

\begin{theorem}
\label{th16}Let $a,b,c\in \Bbb{R}$ and $1\leq p<N,p^{*}<r<\infty $ be given.
Then, $W_{\{a,b\}}^{1,(\infty ,p)}(\Bbb{R}_{*}^{N})\hookrightarrow L^{r}(%
\Bbb{R}^{N};|x|^{c}dx)$ if and only if $a\neq 0$ and one of the following
three conditions holds: \newline
(i) $ap-N\neq b-p$ are on the same side of $-N$ (including $b-p=-N$), $c$ is
in the open interval with endpoints $c^{0}$ and $c^{1}$ and $\theta _{c}\leq 
\frac{p^{*}}{r}$ (i.e., $c$ is in the semi-open interval with endpoints $%
c^{0}$ (not included) and $c^{*1}=\frac{p^{*}}{r}c^{1}+\left( 1-\frac{p^{*}}{%
r}\right) c^{0}$ (included)).\newline
(ii) $ap-N$ and $b-p$ are strictly on opposite sides of $-N,c$ is in the
open interval with endpoints $c^{0}$ and $-N$ and $\theta _{c}\leq \frac{%
p^{*}}{r}$ (i.e., $\theta _{-N}\in (0,1)$ is defined and $c$ is in the open
interval with endpoints $c^{0}$ and $-N$ if $\theta _{-N}\leq \frac{p^{*}}{r}
$ or in the semi-open interval with endpoints $c^{0}$ (not included) and $%
c^{*1}$ (included) if $\theta _{-N}>\frac{p_{*}}{r}$). \newline
(iii) $ap-N=b-p$ and $c=c^{1}$ ($=c^{0}$). If so, there is a constant $C>0$
such that 
\begin{equation}
||u||_{c,r}\leq C||\nabla u||_{b.p}^{\frac{p^{*}}{r}}||\,|x|^{a}u||_{\infty
}^{1-\frac{p^{*}}{r}},\qquad \forall u\in W_{\{a,b\}}^{1,(\infty ,p)}(\Bbb{R}%
_{*}^{N}).  \label{25}
\end{equation}
\end{theorem}

\begin{proof}
The necessity follows from Theorem \ref{th2} and Corollary \ref{cor4} and
the sufficiency is proved below.

(i) By (\ref{23}), the hypothesis $ap-N\neq b-p$ shows that $a^{*}p-N\neq
b^{*}-p.$ Furthermore, since $ap-N$ and $b-p$ are both on the left (or
right) $-N,$ it is readily checked that the same thing is true of $a^{*}p-N$
and $b^{*}-p.$ Another important point is that $b^{*}-p\neq -N.$ Otherwise, $%
b-p+N=-ap\left( \frac{r}{p^{*}}-1\right) .$ Since $r>p^{*},$ this implies $%
b-p+N<0$ if $a>0$ and $b-p+N>0$ if $a<0$ so that, in either case, $ap-N$ and 
$b-p$ are strictly on opposite sides of $-N,$ in contradiction with the
standing hypotheses.

By part (i) of Lemma \ref{lm15}, if $c$ is in the open interval with
endpoints $c^{0}$ and $c^{*1},$ there is a constant $C>0$ such that $%
||v||_{c,p^{*}}\leq C(||\,|x|^{a^{*}}v||_{\infty }+||\nabla v||_{b^{*},p})$
for every $v\in W_{\{a^{*},b^{*}\}}^{1,(\infty ,p)}(\Bbb{R}_{*}^{N}).$ By
Lemma \ref{lm14}, this holds with $v=|u|^{\frac{r}{p^{*}}}$ and $u\in
W_{\{a,b\}}^{1,(\infty ,p)}(\Bbb{R}_{*}^{N}).$ Since (also by Lemma \ref
{lm14}) 
\begin{equation}
||\,|x|^{a^{*}}|u|^{\frac{r}{p^{*}}}||_{\infty }=||\,|x|^{a}u||_{\infty }^{%
\frac{r}{p^{*}}}\text{ and }||\nabla |u|^{\frac{r}{p^{*}}}||_{b^{*},p}\leq 
\frac{r}{p^{*}}||\nabla u||_{b.p}||\,|x|^{a}u||_{\infty }^{\frac{r}{p^{*}}
-1},  \label{26}
\end{equation}
it follows that $||u||_{c,r}\leq C||\,|x|^{a}u||_{\infty }^{1-\frac{p^{*}}{r}%
}(||\,|x|^{a}u||_{\infty }+||\nabla u||_{b.p})^{\frac{p^{*}}{r}}\leq
C(||\,|x|^{a}u||_{\infty }+||\nabla u||_{b.p}),$ whence $W_{\{a,b\}}^{1,(%
\infty ,p)}(\Bbb{R}_{*}^{N})\hookrightarrow L^{r}(\Bbb{R}^{N};|x|^{c}dx).$

Since $b_{*}-p\neq -N,$ the same argument, but with part (iii) of Lemma \ref
{lm15} instead of part (i), shows that $W_{\{a,b\}}^{1,(\infty ,p)}(\Bbb{R}%
_{*}^{N})\hookrightarrow L^{r}(\Bbb{R}^{N};|x|^{c}dx)$ also when $c=c^{*1}.$

(ii) The main difference with (i) is that $ap-N$ and $b-p$ being strictly on
opposite sides of $-N$ does not imply the same thing for $a^{*}p-N$ and $%
b^{*}-p.$ The argument is split into three cases, depending upon the
relative values of $\theta _{-N}$ and $\frac{p^{*}}{r}.$ Observe that since $%
ap-N$ and $b-p$ are strictly on opposite sides of $-N,$ the same thing is
true of $c^{0}$ and $c^{1},$ so that $-N$ is in the open interval with
endpoints $c^{0}$ and $c^{1}$ and $\theta _{-N}\in (0,1)$ is defined.

\textit{Case (ii-{\scriptsize 1}):} $\theta _{-N}>\frac{p^{*}}{r}.$

Since $c^{*1}=\frac{p^{*}}{r}c^{1}+\left( 1-\frac{p^{*}}{r}\right) c^{0}$
with $c^{0}\neq c^{1}$ amounts to $\theta _{c^{*1}}=\frac{p^{*}}{r},$ it
follows that $\theta _{-N}>\theta _{c^{*1}}.$ In particular, $c^{*1}\neq -N$
and so $b^{*}-p\neq -N$ since $c^{*1}=\frac{p^{*}(b^{*}-p+N)}{p}-N,$ a
formula which also shows that $c^{*1}$ and $b^{*}-p$ are on the same side of 
$-N.$ On the other hand, $\theta _{-N}>\theta _{c^{*1}}$ also means that $%
c^{*1}$ and $c^{0}=ar-N$ are on the same side of $-N.$ Since $a\neq 0$ and $%
a^{*}\neq 0$ have the same sign, $ar-N$ and $a^{*}p-N$ are also on the same
side of $-N.$ As a result, $a^{*}p-N$ and $b^{*}-p\neq -N$ are on the same
side of $-N.$ Therefore, the proof of (i) can be repeated verbatim, to the
effect that $W_{\{a,b\}}^{1,(\infty ,p)}(\Bbb{R}_{*}^{N})\hookrightarrow
L^{r}(\Bbb{R}^{N};|x|^{c}dx)$ for every $c$ in the semi-open interval with
endpoints $c^{0}$ (not included) and $c^{*1}$ (included).

\textit{Case (ii-{\scriptsize 2}):} $\theta _{-N}=\frac{p^{*}}{r}.$

Then, $-N=c^{*1}$ and so $b^{*}-p=-N.$ Evidently, this ensures that $%
a^{*}p-N $ and $b^{*}-p$ are still on the same side of $-N$ (though of
course not strictly). The proof of (i) still shows that $W_{\{a,b\}}^{1,(%
\infty ,p)}(\Bbb{R}_{*}^{N})\hookrightarrow L^{r}(\Bbb{R}^{N};|x|^{c}dx)$
for $c$ in the open interval with endpoints $c^{0}$ and $c^{*1}=-N$ (but not
when $c=-N$).

\textit{Case (ii-{\scriptsize 3}): }$\theta _{-N}<\frac{p_{*}}{r}.$

By arguing as in Case (ii-{\scriptsize 1}), $a^{*}p-N$ and $b^{*}-p$ are now
strictly on opposite sides of $-N.$ By combining part (ii) of Lemma \ref
{lm15} and Lemma \ref{lm14} in the same way as in the proof of (i), it
follows that $W_{\{a,b\}}^{1,(\infty ,p)}(\Bbb{R}_{*}^{N})\hookrightarrow
L^{r}(\Bbb{R}^{N};|x|^{c}dx)$ for every $c$ in the open interval with
endpoints $c^{0}$ and $-N.$

(iii) If $ap-N=b-p,$ then $c^{0}=c^{1}=c^{*1},$ the latter since $c^{*1}=%
\frac{p^{*}}{r}c^{1}+\left( 1-\frac{p^{*}}{r}\right) c^{0}$ and $%
c^{0}=c^{1}. $ If $a>0$ ($a<0$), then $b-p>-N$ and $b^{*}-p>-N$ ($b-p<-N$
and $b^{*}-p<-N$). Thus, $a(b^{*}-p+N)>0,$ so that (\ref{25}) follows from
Lemma \ref{lm14}, from (\ref{24}) with $u$ replaced by $|u|^{\frac{r}{p^{*}}%
} $ and from (\ref{26}).
\end{proof}

\section{An example\label{example}}

As an application of Theorem \ref{th1}, we characterize all the real numbers 
$a,b$ and $1\leq p,r<\infty $ such that $W_{\{a,b\}}^{1,(\infty ,p)}(\Bbb{R}%
_{*}^{N})\hookrightarrow L^{r}(\Bbb{R}^{N}),$ i.e., such that $c=0$ is
admissible in Theorem \ref{th1}. We merely give the result and leave the
routine (though somewhat tedious) verification to the reader. This
verification is easier by using Theorems \ref{th8}, \ref{th10}, \ref{th13}
and \ref{th16} rather than Theorem \ref{th1} (equivalent to all four
theorems together, but not phrased in terms of the relative values of $r,p$
and $p^{*}$). It is also helpful to notice that $ap-N$ and $b-p$ are on the
same side of $-N$ (including $-N$) if and only if $a(b-p+N)\geq 0$ and
strictly on opposite sides of $-N$ if and only if $a(b-p+N)<0.$

\begin{theorem}
\label{th17}Let $a,b\in \Bbb{R}$ and $1\leq p<\infty $ be given. Then, $%
W_{\{a,b\}}^{1,(\infty ,p)}(\Bbb{R}_{*}^{N})\hookrightarrow L^{r}(\Bbb{R}%
^{N})$ if and only if $a>0$ and one of the following conditions
holds:\smallskip \newline
(i) $1\leq r<p$ and\newline
(i-{\scriptsize 1}) $a<\frac{N}{r}$ and $b>Np\left( \frac{1}{N}+\frac{1}{r}-%
\frac{1}{p}\right) $ or\newline
(i-{\scriptsize 2}) $a>\frac{N}{r}$ and $b<Np\left( \frac{1}{N}+\frac{1}{r}-%
\frac{1}{p}\right) .\smallskip $\newline
(ii) $p\leq r\leq p^{*}$ and\newline
(ii-{\scriptsize 1}) $a<\frac{N}{r}$ and $b\geq Np\left( \frac{1}{N}+\frac{1%
}{r}-\frac{1}{p}\right) $ or\newline
(ii-{\scriptsize 2}) $a>\frac{N}{r}$ and $b\leq Np\left( \frac{1}{N}+\frac{1%
}{r}-\frac{1}{p}\right) $ or\newline
(ii-{\scriptsize 3}) $a=\frac{N}{r}$ and $b=Np\left( \frac{1}{N}+\frac{1}{r}-%
\frac{1}{p}\right) .\smallskip $\newline
(iii) $p<N,r>p^{*}$ and\newline
(iii-{\scriptsize 1}) $a<\frac{N}{r}$ and $b\geq arp\left( \frac{1}{N}+\frac{%
1}{r}-\frac{1}{p}\right) $ or\newline
(iii-{\scriptsize 2}) $a>\frac{N}{r}$ and $b\leq arp\left( \frac{1}{N}+\frac{%
1}{r}-\frac{1}{p}\right) $ or\newline
(iii-{\scriptsize 3}) $a=\frac{N}{r},b=arp\left( \frac{1}{N}+\frac{1}{r}-%
\frac{1}{p}\right) =Np\left( \frac{1}{N}+\frac{1}{r}-\frac{1}{p}\right)
.\smallskip $
\end{theorem}

For the subspace of radially symmetric functions of $W_{\{a,b\}}^{1,(\infty
,p)}(\Bbb{R}_{*}^{N}),$ part (i) remains true when $0<r<p.$ (Of course, the
embedding properties of radially symmetric functions into $L^{r}(\Bbb{R}%
^{N}) $ have nothing to do with the case $N=1$ of Theorem \ref{th17},
because the reduction to $N=1$ not only requires changing $b$ into $b+N-1$
but also $c$ into $c+N-1,$ which does not preserve the value $c=0$).

Whenever $u\in W_{\{a,b\}}^{1,(\infty ,p)}(\Bbb{R}_{*}^{N})$ and $a<\frac{N}{%
r}$ ($a>\frac{N}{r}$) is assumed, the integrability of $|u|^{r}$ near the
origin (at infinity) is obvious. What is not obvious is that integrability
at infinity (near the origin) is also true when the complementary condition
about $b$ holds. In (ii-{\scriptsize 3}) and (iii-{\scriptsize 3}), $a=\frac{%
N}{r}$ alone does not suffice for $|u|^{r}$ to be integrable near the origin
or at infinity, so that both properties also depend upon the complementary
condition $b=pN\left( \frac{1}{N}+\frac{1}{r}-\frac{1}{p}\right) .$

Below, we give more direct proofs of the sufficiency part of Theorem \ref
{th17} in three simpler special cases. In these proofs, the role played by
the integrability properties of $\nabla u$ (i.e., by the complementary
condition referred to above) becomes more apparent and connections with
(semi) classical results are revealed. On the other hand, the arguments
depend upon the problem at hand and do not suggest a general procedure to
prove Theorem \ref{th17}, let alone Theorem \ref{th1}, in any generality.

\smallskip \textbf{Example 1}. Let $N=1$ and $a=b=1,p=r=1.$ Clearly, the
embedding properties are the same when $\Bbb{R}_{*}$ is replaced by $%
(0,\infty ).$ Thus, by (ii-{\scriptsize 3}), $u\in L^{1}(0,\infty )$ if $%
tu\in L^{\infty }(0,\infty )$ and $tu^{\prime }\in L^{1}(0,\infty ).$ Since
this assumption is unaffected by changing $u$ into $|u|,$ it is not
restrictive to assume $u\geq 0.$ Then, $u$ is locally absolutely continuous
on $(0,\infty )$ and the formula $\int_{\alpha }^{\beta }u=\beta u(\beta
)-\alpha u(\alpha )-\int_{\alpha }^{\beta }tu^{\prime }$ for every $0<\alpha
<\beta <\infty $ shows that, indeed, $u\in L^{1}(0,\infty )$ since the
right-hand side is bounded irrespective of $\alpha $ and $\beta .$

\textbf{Example 2.} More generally, still with $N=1,$ assume $r\geq p$ and $%
b=p\left( 1+\frac{1}{r}-\frac{1}{p}\right) .$ By (ii) of Theorem \ref{th17}, 
$u\in L^{r}(0,\infty )$ if $t^{a}u\in L^{\infty }(0,\infty )$ for some $a>0$
and $t^{\frac{b}{p}}u^{\prime }\in L^{p}(0,\infty ).$ Since $a>0,$ it
follows from part (i) of Lemma \ref{lm6} and from Lemma \ref{lm5} that $%
|u(t)|\leq \int_{\infty }^{t}|u^{\prime }(\tau )|d\tau .$ On the other hand, 
$\left( \int_{0}^{\infty }\left( \int_{\infty }^{t}|u^{\prime }(\tau )|d\tau
\right) ^{r}dt\right) ^{\frac{1}{r}}\leq C\left( \int_{0}^{\infty
}t^{b}|u^{\prime }(t)|^{p}dt\right) ^{\frac{1}{p}}$ by an inequality of
Bradley \cite{Br78}, \cite[p. 40]{Ma85} generalizing the Hardy-type
inequality (\ref{19}), which gives again $u\in L^{r}(0,\infty ).$

\textbf{Example 3.} Other special cases of Theorem \ref{th17} can be given
alternate proofs, including when $N>1.$ For example, if $b=0$ in Theorem \ref
{th17}, then (i-{\scriptsize 1}), (iii-{\scriptsize 2}) and (iii-%
{\scriptsize 3}) cannot occur and the necessary and sufficient conditions
for $W_{\{a,0\}}^{1,(\infty ,p)}(\Bbb{R}_{*}^{N})\hookrightarrow L^{r}(\Bbb{R%
}^{N})$ take the much simpler form:

\begin{corollary}
\smallskip \label{cor18}Let $a\in \Bbb{R}$ and $1\leq p<\infty $ be given.
Then, $W_{\{a,0\}}^{1,(\infty ,p)}(\Bbb{R}_{*}^{N})\hookrightarrow L^{r}(%
\Bbb{R}^{N})$ if and only if $a>0$ and one of the following conditions
holds:\smallskip \newline
(i) $1\leq r<p^{*}$ and $a>\frac{N}{r}.$\newline
(ii) $p<N,r=p^{*}.$\newline
(iii) $p<N,r>p^{*}$ and $a<\frac{N}{r}.$
\end{corollary}

With the help of the right trick, the sufficiency part of Corollary \ref
{cor18} can also be proved by classical arguments (necessity still relies on
Section \ref{necessary}). Since $N=1$ is trivial, we assume $N\geq 2$ and,
for brevity, we only show how $W_{\{a,0\}}^{1,(\infty ,p)}(\Bbb{R}%
_{*}^{N})\subset L^{r}(\Bbb{R}^{N})$ can be recovered without elaborating on
the continuity issue.

We shall need the preliminary remark that $W_{\{a,0\}}^{1,(\infty ,p)}(\Bbb{R%
}_{*}^{N})=W_{\{a,0\}}^{1,(\infty ,p)}(\Bbb{R}^{N})$ for every $a\in \Bbb{R}$
when $N\geq 2.$ To see this, note that $\Bbb{R}_{*}^{N}$ is the union of
finitely many open half-spaces $H_{j}.$ If $u\in W_{\{a,0\}}^{1,(\infty ,p)}(%
\Bbb{R}_{*}^{N}),$ then $\nabla u\in L^{p}(H_{j})$ and so $u\in
W^{1,p}(B\cap H_{j})$ where $B$ is the unit ball of $\Bbb{R}^{N}.$ Since $j$
is arbitrary, it follows that $u\in W^{1,p}(B\backslash \{0\}).$ By 
\cite[Theorem 2.44]{HeKiMa93} or by the so-called ``absolutely continuous''
characterization of Sobolev functions (\cite[Theorem 2.1.4]{Zi89}), $%
W^{1,p}(B\backslash \{0\})=W^{1,p}(B)$ since $N\geq 2.$ In particular, $%
\nabla u$ as a distribution on $B$ coincides with $\nabla u$ as a
distribution on $B\backslash \{0\}.$ Thus, $\nabla u$ is the same as a
distribution on $\Bbb{R}^{N}$ or $\Bbb{R}_{*}^{N},$ so that $u\in
W_{\{a,0\}}^{1,(\infty ,p)}(\Bbb{R}^{N}).$

\begin{remark}
\label{rm4}By a similar argument, $W_{\{a,b\}}^{1,(q,p)}(\Bbb{R}%
_{*}^{N})=W_{\{a,b\}}^{1,(q,p)}(\Bbb{R}^{N})$ when $N\geq 2$ and $b\leq 0,$
irrespective of $1\leq p<\infty ,1\leq q\leq \infty $ and $a\in \Bbb{R}.$
However, equality need not hold when $b\geq 0.$ For example, if $u\in
C^{\infty }(\Bbb{R}_{*}^{N})$ has bounded support and coincides with $%
|x|^{-N}$ on a neighborhood of $0,$ then $u$ is not integrable near $0.$
Thus, $u$ is not in $W_{\{a,b\}}^{1,(q,p)}(\Bbb{R}^{N})$ for any values of
the parameters, but $u\in W_{\{a,b\}}^{1,(q,p)}(\Bbb{R}_{*}^{N})$ if $a$ and 
$b$ are large enough (depending on $p$ and $q$).
\end{remark}

By using $W_{\{a,0\}}^{1,(\infty ,p)}(\Bbb{R}_{*}^{N})=W_{\{a,0\}}^{1,(%
\infty ,p)}(\Bbb{R}^{N}),$ a ``direct'' proof that $a>0$ plus any one of the
conditions (i) to (iii) of Corollary \ref{cor18} suffices for $%
W_{\{a,0\}}^{1,(\infty ,p)}(\Bbb{R}_{*}^{N})\subset L^{r}(\Bbb{R}^{N})$ goes
as follows.

Suppose first that $p<N$ and let $u\in W_{\{a,0\}}^{1,(\infty ,p)}(\Bbb{R}
^{N})$ be given. By Sobolev's inequality, there is a constant $C>0$
independent of $u$ such that $||u-c||_{p^{*}}\leq C||\nabla u||_{p}$ for
some $c\in \Bbb{R}.$ In addition, $|x|^{a}u\in L^{\infty }(\Bbb{R}^{N})$
with $a>0$ implies that $u$ tends (essentially) uniformly to $0$ at
infinity. Thus, $c=0,$ so that $u\in L^{p^{*}}(\Bbb{R}^{N})$ and $%
||u||_{p^{*}}\leq C||\nabla u||_{p}.$ Not only this shows that (ii)
suffices, but it also corroborates the inequality (\ref{8}) (or (\ref{9}))
of Theorem \ref{th1} when $a\neq \frac{N}{p}-1$ (or $a=\frac{N}{p}-1$), $b=0$
and $r=p^{*}$ (so that $c^{1}=0$ in (\ref{4}), whence $\theta _{0}=1$).

Next, if $r>p^{*},$ the same argument that $u$ tends to $0$ at infinity and $%
u\in L^{p^{*}}(\Bbb{R}^{N})$ show that $|u|^{r}$ is integrable at infinity.
If $a<\frac{N}{r},$ then $|u|^{r}$ is also integrable near $0,$ which proves
that (iii) suffices. On the other hand, if $r<p^{*},$ then $u\in L^{p^{*}}(%
\Bbb{R}^{N})$ implies $u\in L_{loc}^{r}(\Bbb{R}^{N})$ while $|u|^{r}$ is
integrable at infinity if $a>\frac{N}{r}.$ This proves that (i) suffices
when $p<N.$

It only remains to prove the sufficiency of (i) when $p\geq N$ and $1\leq
r<\infty .$ As before, $a>\frac{N}{r}$ implies that $|u|^{r}$ is integrable
at infinity. Once again, let $B$ denote the unit ball of $\Bbb{R}^{N}$ and $%
H_{j}$ a finite collection of open hyperplanes such that $\cup H_{j}=\Bbb{R}
_{*}^{N}.$ The condition $\nabla u\in L^{p}(\Bbb{R}^{N})$ implies $\nabla
u\in L^{p}(B\cap H_{j}).$ Therefore, $u\in W^{1,p}(B\cap H_{j})$ and so $%
u\in L^{r}(B\cap H_{j})$ by the Sobolev embedding theorem. As a result, $%
u\in L^{r}(B\backslash \{0\})=L^{r}(B)$ and the proof is complete.

\end{document}